\documentclass[12pt]{amsart}
\usepackage{graphicx}
\usepackage{amssymb}
\usepackage{amscd}
\usepackage{cite}
\usepackage[all]{xy}
\usepackage{url}
\newtheorem{tw}{Theorem}[section]

\theoremstyle{remark}
\newtheorem{uw}[tw]{Remark}

\theoremstyle{definition}

\newcommand{\cal}[1]{\mathcal{#1}}

\newcommand{\ro}{\varrho}

\newcommand{\kre}[1]{\overline{#1}}
\newcommand{\ukre}[1]{\underline{#1}}

\newcommand{\map}[3]{#1\colon #2\to #3}

\newcommand{\cM}{{\cal M}}
\newcommand{\cT}{{\cal T}}

\newcommand{\SR}[1]{\underset{\rightarrow}{[#1]}}
\newcommand{\SL}[1]{\underset{\leftarrow}{[#1]}}

\begin{document}

\numberwithin{equation}{section}

\title[A finite presentation for the twist subgroup\ldots]
{A finite presentation for the twist subgroup of the mapping class group of a nonorientable surface}

\author{Micha\l\ Stukow}

\thanks{Supported by NCN grant 2012/05/B/ST1/02171.}

\address[]{
Institute of Mathematics, University of Gda\'nsk, Wita Stwosza 57, 80-952 Gda\'nsk, Poland }

\email{trojkat@mat.ug.edu.pl}



\begin{abstract}
Let $N_{g,s}$ denote the nonorientable surface of genus $g$ with $s$ boundary components. Recently Paris and Szepietowski
\cite{SzepParis} obtained an explicit finite presentation for the mapping class group $\cM(N_{g,s})$ of the surface
$N_{g,s}$, where $s\in\{0,1\}$ and $g+s>3$. Following this work, we obtain a finite presentation for the subgroup
$\cT(N_{g,s})$ of $\cM(N_{g,s})$ generated by Dehn twists.
\end{abstract}

\maketitle%
\section{Introduction}%
Let $N_{g,s}$ be a smooth, nonorientable, compact surface of genus $g$ with $s$ boundary components. If
$s$ is zero, then we omit it from the notation. If we do not want to emphasise the numbers $g,s$, we simply 
write $N$ for a surface $N_{g,s}$. Recall that $N_{g}$ is a connected sum of $g$ projective planes 
and $N_{g,s}$ is obtained from $N_g$ by removing $s$ open disks.

Let ${\textrm{Diff}}(N)$ be the group of all diffeomorphisms $\map{h}{N}{N}$ such that $h$ is the identity 
on each boundary component. By ${\cal{M}}(N)$ we denote the quotient group of ${\textrm{Diff}}(N)$ by
the subgroup consisting of maps isotopic to the identity, where we assume that isotopies are 
the identity on each boundary component. ${\cal{M}}(N)$ is called the \emph{mapping class group} of $N$. 

The mapping class group ${\cal{M}}(S_{g,s})$ of an orientable surface is defined analogously, but we consider only
orientation preserving maps. 
\subsection{Background}
One of the most important elements in mapping class groups of surfaces are Dehn twists. They were discovered 
by Max Dehn, who first observed that they generate the mapping class group ${\cal{M}}(S_g)$ 
of a closed oriented surface $S_g$. Twists were rediscovered by Lickorish \cite{Lick1,Lick2}, who also 
proved that ${\cal{M}}(S_g)$ is generated by $3g-1$ Dehn twists about nonseparating circles. Later Humphries reduced this generating set to $2g+1$ twists \cite{Hump}. 

Since Dehn twists generate the mapping class group ${\cal{M}}(S_g)$, it is natural to ask about possible relations between them. Let us mention some results in this direction. Birman \cite{Bir2} observed that there is a close relation between mapping class group ${\cal{M}}(S_g)$ and the mapping class group of a punctured sphere, which in fact is a quotient of the braid group $B_{2g+2}$. This correspondence leads to a number of interesting relations, for example: \emph{braid} and \emph{chain relations}, relations with hyperelliptic involution, relations with elements of finite order. Later Johnson \cite{John1} discovered the so-called \emph{lantern relation}, which apparently has been used by Dehn in 1920's. It turned out that this set of relations was enough to give a full presentation of ${\cal{M}}(S_g)$, which was obtained by Wajnryb \cite{Wajn_pre}. Later some other relations were discovered, for example \emph{star relations} or relations between fundamental elements in Artin groups embedded in ${\cal{M}}(S_g)$. These relations led to some other interesting presentations of ${\cal{M}}(S_g)$ -- see \cite{Gervais_top,MatsumotoPres}.

In the nonorientable case, Lickorish \cite{Lick3} first observed that Dehn twists do not generate 
the mapping class group ${\cal{M}}(N_{g})$ for $g\geq 2$. More precisely, he proved that Dehn twists 
generate the so-called \emph{twist subgroup} ${\cal{T}}(N_g)$ which is of index 2 in 
${\cal{M}}(N_{g})$. Later Chillingworth \cite{Chil} found finite generating sets for 
${\cal{T}}(N_{g})$ and ${\cal{M}}(N_{g})$. These generating sets were extended to the case of a surface
with punctures and/or boundary components in \cite{Kork-non, Stukow_SurBg, Stukow_HomTw}.

As for relations, recently Paris and Szepietowski \cite{SzepParis} obtained a finite presentations for groups 
${\cal{M}}(N_{g,s})$ where $s\in\{0,1\}$ and $g+s>3$.
\subsection{Main results}
The main goal of this paper is to find a complete set of relations between Dehn twists on a nonorientable surface $N$. To be more precise, we obtain a presentation for the twist subgroup 
${\cal{T}}(N_{g,s})$ of the mapping class group ${\cal{M}}(N_{g,s})$ of a nonorientable surface (Theorems \ref{PresTwist1} and \ref{PresTwist2}), where $s\in\{0,1\}$ and $g+s>3$. The obtained presentations may seem to be complicated, but many relations are needed only for small genera and stably the presentations are quite simple.

Our starting point is the presentation of ${\cal{M}}(N_{g,s})$ obtained by Paris and Szepietowski \cite{SzepParis}, 
however their presentation has $g-1$ generators which are not elements of ${\cal{T}}(N_{g,s})$, hence 
it leads to a very complicated presentation of the twist subgroup. Therefore, we use a recent simplification of their presentation \cite{StukowSimpSzepPar}, which has only one generator 
not belonging to ${\cal{T}}(N_{g,s})$ (Theorems \ref{SimParSzep1}, \ref{SimParSzep2} and \ref{Uw:F}). 
\section{Preliminaries}
\subsection{Notation}
Let us represent surfaces $N_{g,0}$ and $N_{g,1}$ as respectively a sphere or a disc with $g$ crosscaps and 
let $\alpha_1,\ldots,\alpha_{g-1},\beta$ be two-sided circles indicated in Figure~\ref{r01}. 
\begin{figure}[h]
\begin{center}
\includegraphics[width=0.95\textwidth]{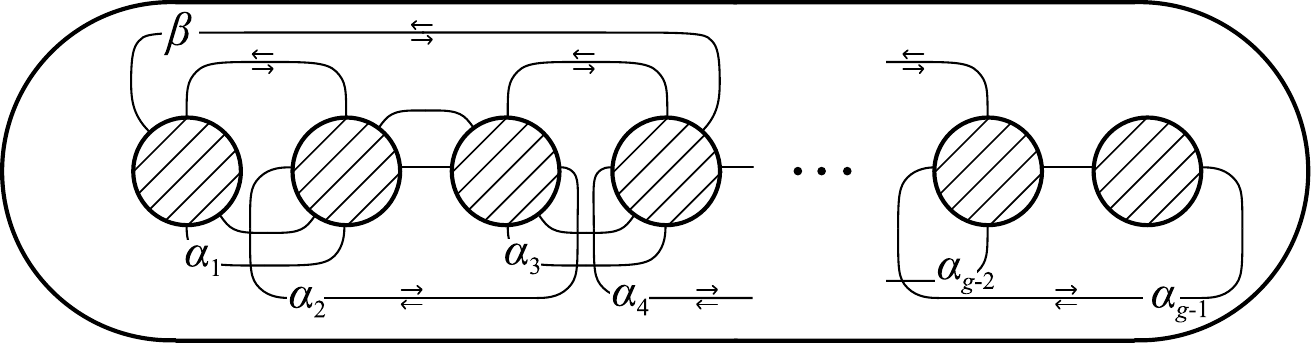}
\caption{Surface $N$ as a sphere/disc with crosscaps.}\label{r01} %
\end{center}
\end{figure}
Small arrows in this figure indicate directions of Dehn twists $a_1,\ldots,a_{g-1},b$ associated with these circles. 
Observe that $\beta$ (hence also $b$) is defined only if $g\geq 4$. From now on whenever we use $b$, we silently assume that
$g\geq 4$.

Moreover, for any unoriented one-sided circle $\mu$ and oriented two-sided circle $\alpha$ which intersects $\mu$ in one point (Figure \ref{r03}), we define a \emph{crosscap slide} (or Y-homeomorphism) $Y_{\mu,\alpha}$, that is the effect of pushing $\mu$ along the curve $\alpha$ -- for precise definition see Section 2.2 of \cite{SzepParis}. 
\begin{figure}[h]
\begin{center}
\includegraphics[width=0.7\textwidth]{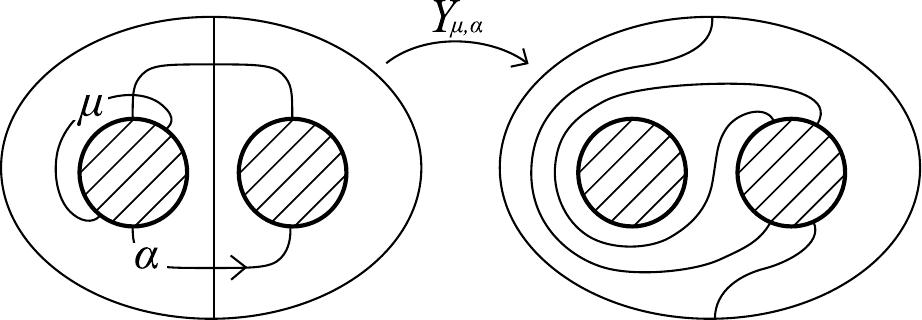}
\caption{Crosscap slide.}\label{r03} %
\end{center}
\end{figure}
In particular, let $y=Y_{\mu_1,\alpha_1}$, where $\mu_1,\alpha_1$ are curves indicated in Figure \ref{r04}. 
\begin{figure}[h]
\begin{center}
\includegraphics[width=0.9\textwidth]{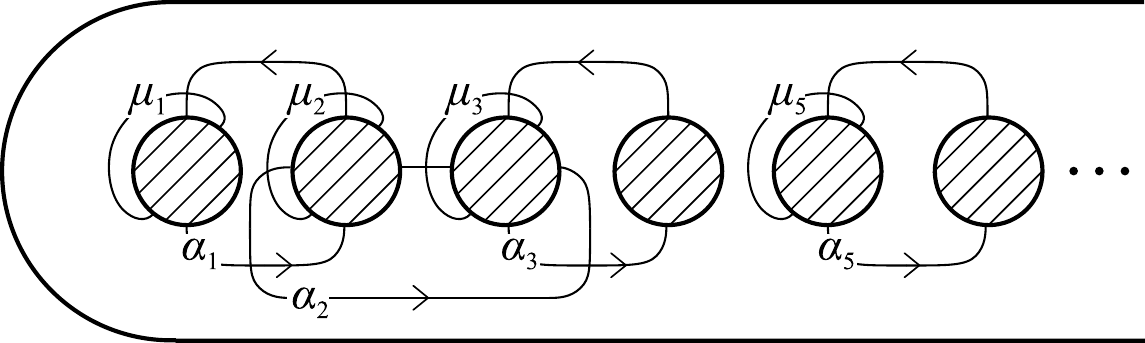}
\caption{Circles $\mu_i$ and $\alpha_i$.}\label{r04} %
\end{center}
\end{figure}

The following three theorems are the main results of \cite{StukowSimpSzepPar}
\begin{tw} \label{SimParSzep1}
 If $g\geq 3$ is odd or $g=4$, then ${\cal{M}}(N_{g,1})$ admits a presentation with generators $a_1,\ldots,a_{g-1},y$ and $b$
for
$g\geq 4$. The defining relations are 
\begin{enumerate} 
\item[(A1)] $a_ia_j=a_ja_i$\quad for $g\geq 4$, $|i-j|>1$,
\item[(A2)] $a_ia_{i+1}a_i=a_{i+1}a_ia_{i+1}$\quad for $i=1,\ldots,g-2$,
 \item[(A3)] $a_ib=ba_i$\quad for $g\geq 4,i\neq 4$,
 \item[(A4)] $ba_4b=a_4ba_4$\quad for $g\geq 5$,
 \item[(A5)] $(a_2a_3a_4b)^{10}=(a_1a_2a_3a_4b)^6$\quad for $g\geq 5$,
 \item[(A6)] $(a_2a_3a_4a_5a_6b)^{12}=(a_1a_2a_3a_4a_5a_6b)^{9}$\quad for $g\geq 7$,
 \item[(B1)] $y(a_2a_3a_1a_2ya_2^{-1}a_1^{-1}a_3^{-1}a_2^{-1})=(a_2a_3a_1a_2ya_2^{-1}a_1^{-1}a_3^{-1}a_2^{-1})y$
\quad for $g\geq 4$,
 \item[(B2)]  $y(a_2a_1y^{-1}a_2^{-1}ya_1a_2)y=a_1(a_2a_1y^{-1}a_2^{-1}ya_1a_2)a_1$,
\item[(B3)] $a_iy=ya_i$ for $g\geq 4$, $i=3,4,\ldots,g-1$,
\item[(B4)] $a_2(ya_2y^{-1})=(ya_2y^{-1})a_2$,
\item[(B5)] $ya_1=a_1^{-1}y$,
\item[(B6)] $byby^{-1}=[a_1a_2a_3(y^{-1}a_2y)a_3^{-1}a_2^{-1}a_1^{-1}][a_2^{-1}a_3^{-1}(ya_2y^{-1})a_3a_2]$
\quad for $g\geq 4$,
\item[(B7)] $(a_4a_5a_3a_4a_2a_3a_1a_2ya_2^{-1}a_1^{-1}a_3^{-1}a_2^{-1}a_4^{-1}a_3^{-1}a_5^{-1}a_4^{-1})b=\\
b(a_4a_5a_3a_4a_2a_3a_1a_2ya_2^{-1}a_1^{-1}a_3^{-1}a_2^{-1}a_4^{-1}a_3^{-1}a_5^{-1}a_4^{-1})$
\quad for $g\geq 6$,
\item[(B8)]
$[(ya_1^{-1}a_2^{-1}a_3^{-1}a_4^{-1})b(a_4a_3a_2a_1y^{-1})][(a_1^{-1}a_2^{-1}a_3^{-1}a_4^{-1})b^{-1}(a_4a_3a_2a_1)]=\\
\ [(a_4^{-1 } a_3^{-1}a_2^{-1})y(a_2a_3a_4)][a_3^{-1}a_2^{-1}y^{-1}a_2a_3]
[a_2^{-1}ya_2]y^{-1}$ \quad for $g\geq~5$.
\end{enumerate}
If $g\geq 6$ is even, then ${\cal{M}}(N_{g,1})$ admits a presentation with generators
$a_1,\ldots,a_{g-1}$, $y$, $b$ and additionally $b_0,b_1,\ldots,b_{\frac{g-2}{2}}$. The defining relations 
are relations (A1)--(A6), (B1)--(B8) above and additionally
\begin{enumerate}
 \item[(A7)] $b_0=a_1, b_1=b$,
\item[(A8)] $b_{i+1}=(b_{i-1}a_{2i}a_{2i+1}a_{2i+2}a_{2i+3}b_{i})^5(b_{i-1}a_{2i}a_{2i+1}a_{2i+2}a_{2i+3})^{-6}$\quad\\	  for
$1\leq i\leq \frac{g-4}{2}$,
\item[(A9a)] $b_2b=bb_2$\quad for $g=6$,
\item[(A9b)] $b_{\frac{g-2}{2}}a_{g-5}=a_{g-5}b_{\frac{g-2}{2}}$\quad for $g\geq 8$. \qed
\end{enumerate}
\end{tw}
\begin{tw} \label{SimParSzep2}
 If $g\geq 4$, then the group ${\cal M}(N_{g,0})$ is isomorphic to the quotient of the group ${\cal M}(N_{g,1})$ with
presentation given in Theorem \ref{SimParSzep1} obtained by adding a generator $\ro$ and relations
\begin{itemize}
 \item[(C1a)] $(a_1a_2\cdots a_{g-1})^g=\ro$\quad for $g$ odd,
 \item[(C1b)] $(a_1a_2\cdots a_{g-1})^g=1$\quad for $g$ even,
\item[$(\text{C2}_1)$] $\ro a_1=a_1\ro$,
\item[(C3)] $\ro^2=1$,
\item[(C4a)] $(y^{-1}a_2a_3\cdots a_{g-1}ya_2a_3\cdots a_{g-1})^{\frac{g-1}{2}}=1$\quad for $g$ odd,
\item[(C4b)] $(y^{-1}a_2a_3\cdots a_{g-1}ya_2a_3\cdots a_{g-1})^{\frac{g-2}{2}}y^{-1}a_2a_3\cdots a_{g-1}=\ro$\quad for $g$ even. \qed
\end{itemize}
\end{tw}
\begin{tw}\label{Uw:F}
 Relations (C4a), (C4b) and $\text{(C2}_\text{1}\text{)}$ in the presentation given by Theorem \ref{SimParSzep2} may be replaced by 
\begin{itemize}
\item[(C2)] $\ro a_i=a_i\ro$\quad for $i=1,\ldots,g-1$,
\item[(C5)] $y\ro=\ro y^{-1}$,
 \item[(C4)] $(y\ro a_2a_3\cdots a_{g-1})^{g-1}=1$. \qed
\end{itemize}
\end{tw}
\section{Presentation for the twist subgroup}
Recall that for $s\leq 1$ and $g\geq 3$ the twist subgroup ${\cal{T}}(N_{g,s})$ has index 2 in ${\cal{M}}(N_{g,s})$
(for details see \cite{Lick3,Stukow_HomTw}), 
hence we can obtain its presentation using Reidemeister--Schreier rewriting process. To be more precise,
we define a Schreier transversal $U=\{1,y\}$ for ${\cal{T}}(N_{g,s})$ in ${\cal{M}}(N_{g,s})$ and 
for any $h\in {\cal{M}}(N_{g,s})$ we define
\[\kre{h}=\begin{cases}
           1&\text{if $h\in {\cal{T}}(N_{g,s})$}\\
           y&\text{if $h\not\in {\cal{T}}(N_{g,s})$.}
          \end{cases}
\]
The Reidemeister--Schreier theorem states that ${\cal{T}}(N_{g,s})$ admits a presentation with 
generators $ux\kre{ux}^{-1}$, where $x$ is a generator of ${\cal{M}}(N_{g,s})$, $u\in U$ 
and $ux\not \in U$. The set of defining relations consists of relations of the form $uru^{-1}$, where $u\in U$ and 
$r$ is a defining relation for ${\cal{M}}(N_{g,s})$.
\begin{tw} \label{PresTwist1}
 If $g\geq 3$ is odd or $g=4$, then ${\cal{T}}(N_{g,1})$ admits a presentation with generators 
 $a_1,\ldots,a_{g-1},e,f,y^2$ and $b,c$ for
$g\geq 4$. The defining relations are (A1)--(A6) and 
\begin{itemize}
 \item[$(\kre{\text{A1}}_1)$] $ea_j=a_je$\quad for $g\geq 5$, $j\geq 4$,
 \item[$(\kre{\text{A1}}_2)$] $fa_j=a_jf$\quad for $g\geq 5$, $j\geq 4$,
 \item[$(\kre{\text{A2}}_1)$] $a_1ea_1=ea_1e$,
 \item[$(\kre{\text{A2}}_2)$] $a_3^{-1}ea_3^{-1}=ea_3^{-1}e$\quad for $g\geq 4$,
 \item[$(\kre{\text{A2}}_3)$] $a_1fa_1=fa_1f$,
 \item[$(\kre{\text{A3}}_1)$] $a_1c=ca_1$\quad for $g=4,5$,
 \item[$(\kre{\text{A3}}_2)$] $ec=ce$\quad for $g=4,5$,
 \item[$(\kre{\text{A4}})$] $ca_4c=a_4ca_4$\quad for $g=5,6$,
 \item[$(\kre{\text{A5}})$] $(e^{-1}a_3a_4c)^{10}=(a_1^{-1}e^{-1}a_3a_4c)^6$\quad for $g=5,6$,
 \item[$(\kre{\text{A6}})$] $(e^{-1}a_3a_4a_5a_6c)^{12}=(a_1^{-1}e^{-1}a_3a_4a_5a_6c)^{9}$\quad for $g=7,8$,
 \item[$(\kre{\text{B1}})$] $(a_2a_3a_1a_2ea_1a_3^{-1}e)(a_2a_3a_1a_2fa_1a_3^{-1}f)=1$\quad for $g\geq 4$, 
 \item[$(\kre{\text{B2}}_1)$] $y^2=a_2a_1ea_1a_2a_1a_2a_1a_2fa_1a_2$,
 \item[$(\kre{\text{B2}}_2)$] $(a_2a_1ea_1a_2a_1a_2a_1a_2fa_1a_2)(a_2a_1fa_1a_2a_1a_2a_1a_2ea_1a_2)=1$,
 \item[$(\kre{\text{B3}})$] $y^2a_3=a_3y^2$\quad for $g\geq 4$,
 \item[$(\kre{\text{B4}}_1)$] $ea_2=a_2e$,
 \item[$(\kre{\text{B4}}_2)$] $fa_2=a_2f$,
 \item[$(\kre{\text{B6}}_1)$] $bc=[a_1a_2a_3f^{-1}a_3^{-1}a_2^{-1}a_1^{-1}][a_2^{-1}a_3^{-1}e^{-1}a_3a_2]$\quad for $g\geq 4$,
 \item[$(\kre{\text{B6}}_2)$] $c(y^2by^{-2})=[a_1^{-1}e^{-1}a_3a_2a_3^{-1}ea_1][ea_3^{-1}(y^2a_2y^{-2})a_3e^{-1}]$\quad for
$g=4,5$,
 \item[$(\kre{\text{B7}}_1)$] $(a_4a_5a_3a_4a_2a_3a_1a_2ea_1a_3^{-1}ea_4^{-1}a_3^{-1}a_5^{-1}a_4^{-1})c=\\
b(a_4a_5a_3a_4a_2a_3a_1a_2ea_1a_3^{-1}ea_4^{-1}a_3^{-1}a_5^{-1}a_4^{-1})$\quad for $g\geq 6$,
\item[$(\kre{\text{B7}}_2)$] $(a_2^{-1}a_1^{-1}a_3^{-1}a_2^{-1}a_4^{-1}a_3^{-1}a_5^{-1}a_4^{-1})b
(a_4a_5a_3a_4a_2a_3a_1a_2)y^2=\\y^2(a_2^{-1}a_1^{-1}a_3^{-1}a_2^{-1}a_4^{-1}a_3^{-1}a_5^{-1}a_4^{-1})b
(a_4a_5a_3a_4a_2a_3a_1a_2)$\quad for $g\geq 6$,
\item[$(\kre{\text{B8}}_1)$]
$\left[(a_1ea_3^{-1}a_4^{-1})c(a_4a_3e^{-1}a_1^{-1})\right]\left[(a_1^{-1}a_2^{-1}a_3^{-1}a_4^{-1})b^{-1}(a_4a_3a_2a_1)\right
] =\\
a_4^{-1}\left[(a_3^{-1}a_2^{-1}e^{-1}a_3)a_4(a_3^{-1}ea_2a_3)\right]a_2^{-1}e^{-1}$\quad for $g\geq 5$,
\item[$(\kre{\text{B8}}_2)$]
$\left[(a_1^{-1}a_2^{-1}a_3^{-1}a_4^{-1})b(a_4a_3a_2a_1)\right]\left[(a_1fa_3^{-1}a_4^{-1})y^{-2}c^{-1}y^2(a_4a_3f^{-1}a_1^{
-1 } )
\right]=\\
a_4^{-1}\left[(a_3^{-1}fa_2a_3)a_4(a_3^{-1}a_2^{-1}f^{-1}a_3)\right]
fa_2$\quad for \mbox{$g=5,6$.}
\end{itemize}
If $g\geq 6$ is even, then ${\cal{T}}(N_{g,1})$ admits a presentation with generators
$a_1,\ldots,a_{g-1}$, $e,f,y^2,b,c$ and additionally $b_0,b_1,\ldots,b_{\frac{g-2}{2}}$,
$\kre{b}_{\frac{g-6}{2}}, \kre{b}_{\frac{g-4}{2}}, \kre{b}_{\frac{g-2}{2}}$. The defining
relations 
are relations (A1)--(A9), ($\kre{\text{A1}}_1$)--($\kre{\text{A6}}$),
($\kre{\text{B1}}$)--($\kre{\text{B8}}_2$) and additionally
\begin{enumerate}
 \item[($\kre{\text{A7a}}$)] $\kre{b}_0=a_1^{-1}$, $\kre{b}_1=c$\quad for $g=6$,
 \item[($\kre{\text{A7b}}$)] $\kre{b}_1=c$\quad for $g=8$,
 \item[($\kre{\text{A7c}}$)] $\kre{b}_{i}=z_{g-1}b_iz_{g-1}^{-1}$\quad where $i=\frac{g-6}{2},\frac{g-4}{2}$, $i\geq 2$ 
 and 
 \[z_{g-1}=(a_{g-1}a_ga_{g-2}a_{g-1}\cdots a_3a_4e^{-1}a_3a_1^{-1}e^{-1})(a_2^{-1}a_1^{-1}\cdots
a_{g-1}^{-1}a_{g-2}^{-1}a_{g}^{-1}a_{g-1}^{-1}),\]
\item[($\kre{\text{A8a}}$)]
$\kre{b}_{2}=(\kre{b}_0e^{-1}a_{3}a_{4}a_{5}\kre{b}_{1})^5(\kre{b}_{0}e^{-1}a_{3}a_{4}a_{5})^{-6}$\quad for $g=6$,
\item[($\kre{\text{A8b}}$)]
$\kre{b}_{\frac{g-2}{2}}=(\kre{b}_{\frac{g-6}{2}}a_{g-4}a_{g-3}a_{g-2}a_{g-1}\kre{b}_{\frac{g-4}{2}})^5(\kre{b}_{\frac{g-6}{2}}
a_{g-4}a_{g-3}a_{g-2}a_{g-1})^{-6}$\quad for \mbox{$g\geq 8$},
\item[($\kre{\text{A9a}}$)] $\kre{b}_2c=c\kre{b}_2$\quad for $g=6$,
\item[($\kre{\text{A9b}}$)] $\kre{b}_{\frac{g-2}{2}}a_{g-5}=a_{g-5}\kre{b}_{\frac{g-2}{2}}$\quad for $g\geq
8$.
\end{enumerate}
\end{tw}
\begin{proof}
 As noted before, we apply Reidemeister--Schreier theorem to the presentation given by Theorem 
 \ref{SimParSzep1}. Hence as generators of the twist subgroup ${\cal{T}}(N_{g,1})$ we obtain 
 $a_1,\ldots,a_{g-1},ya_1y^{-1},\ldots, ya_{g-1}y^{-1}, y^2$ and $b,yby^{-1}$ for $g\geq 4$. Moreover, if 
 $g\geq 6$ is even, we have additional generators: $b_0,b_1,\ldots,b_{\frac{g-2}{2}}$, 
 $yb_0y^{-1},yb_1y^{-1},\ldots,yb_{\frac{g-2}{2}}y^{-1}$. Let us name some of these generators: 
 \[e=ya_2^{-1}y^{-1},\ c=yby^{-1},\ \kre{b}_i=yb_iy^{-1}\ \text{for  $i=0,\ldots,\frac{g-2}{2}$}.\]
 We also add one generator $f=y^{-1}a_2^{-1}y$ with defining relation
 \begin{itemize}
  \item[(D1)] $f=y^{-2}ey^2$
 \end{itemize}
 (see Figure \ref{r07}).\\
 \begin{figure}[h]
\begin{center}
\includegraphics[width=1\textwidth]{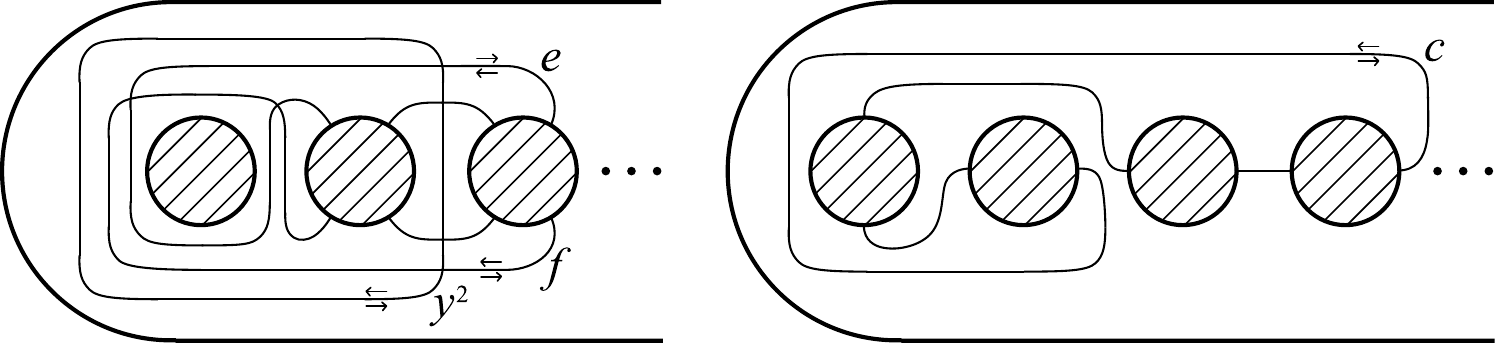}
\caption{Twists $e,f,y^2,c$.}\label{r07} %
\end{center}
\end{figure}
 {\bf (B3)} Observe first that relation (B3) rewrites as
 \[ya_iy^{-1}=a_i\quad\text{for $i=3,4,\ldots,g-1$}.\]
 This means that we can remove generators $ya_3y^{-1},\ldots,ya_{g-1}y^{-1}$ from the presentation, hence
 from now on we will silently identify $ya_iy^{-1}$ with $a_i$ for $i=3,4,\ldots,g-1$.\\
 {\bf (B5)} Similarly, (B5) allows us to identify $ya_1y^{-1}$ with $a_1^{-1}$.
 
 Observe also that conjugations of (B3) and (B5) by $y$ give 
 \begin{itemize}
  \item[$(\kre{\text{B3}})$] $y^2a_i=a_iy^2$\quad for $i=1,3,4,\ldots,g-1$.
 \end{itemize}
We will show later that this relation can be replaced by ($\kre{\text{A1}}_2$) if $i\neq 3$.\\
 {\bf (A1)--(A9)} Relations which do not contain $y$, that is (A1)--(A9) does not need rewriting, however we need to add
their 
 versions conjugated by $y$. This gives relations ($\kre{\text{A1}}_1$), ($\kre{\text{A2}}_1$), ($\kre{\text{A2}}_2$),
 ($\kre{\text{A3}}_2$), ($\kre{\text{A4}}$)--($\kre{\text{A6}}$) and
 \begin{itemize}
  \item[($\kre{\text{A3}}_1$)] $a_ic=ca_i$\quad for $g\geq 4$, $i\neq 2,4$.
 \end{itemize}
 If $g\geq 6$ is even, then we have additionally
 \begin{itemize}
  \item[($\kre{\text{A7}}$)] $\kre{b}_0=a_1^{-1}$, $\kre{b}_1=c$,
\item[($\kre{\text{A8a}}$)] $\kre{b}_{2}=(\kre{b}_0e^{-1}a_{3}a_{4}a_{5}\kre{b}_{1})^5(\kre{b}_{0}e^{-1}a_{3}a_{4}a_{5})^{-6}$,
\item[($\kre{\text{A8b}}$)] 
$\kre{b}_{i+1}=(\kre{b}_{i-1}a_{2i}a_{2i+1}a_{2i+2}a_{2i+3}\kre{b}_{i})^5(\kre{b}_{i-1}a_{2i}a_{2i+1}a_{2i+2}a_{2i+3})^{-6}$ \quad\\	  for
$2\leq i\leq \frac{g-4}{2}$,
\item[($\kre{\text{A9a}}$)] $\kre{b}_2c=c\kre{b}_2$\quad for $g=6$,
\item[($\kre{\text{A9b}}$)] $\kre{b}_{\frac{g-2}{2}}a_{g-5}=a_{g-5}\kre{b}_{\frac{g-2}{2}}$\quad for $g\geq 8$.
  \end{itemize} 
{\bf (B4)} Relation (B4) and its conjugation by $y^{-1}$ rewrite respectively as $(\kre{\text{B4}}_1)$ and
$(\kre{\text{B4}}_2)$. It is also useful to note that relations (D1), $(\kre{\text{B3}})$,
$(\kre{\text{A2}}_1)$ and $(\kre{\text{A2}}_2)$ imply that
\begin{itemize}
 \item[$(\kre{\text{A2}}_3)$] $a_1fa_1=fa_1f$,
 \item[$(\kre{\text{A2}}_4)$] $a_3^{-1}fa_3^{-1}=fa_3^{-1}f$\quad for $g\geq 4$.
\end{itemize}
{\bf (B2)} Using (A2), $(\kre{\text{A2}}_1)$, $(\kre{\text{A2}}_2)$ and $(\kre{\text{B4}}_1)$ we rewrite (B2).
\[\begin{aligned}
   &\SR{\ukre{y}}(a_2a_1\ukre{y^{-1}}a_2^{-1}\SR{y}a_1a_2)y=a_1(a_2a_1y^{-1}a_2^{-1}ya_1a_2)a_1,\\
   &\SL{e^{-1}a_1^{-1}a_2^{-1}a_1^{-1}e^{-1}}y^2=a_1a_2a_1f[a_1a_2a_1],\\
   &y^2=e[a_1a_2a_1]ea_1a_2a_1fa_2a_1a_2,\\
   &y^2=\SR{e}a_2a_1a_2\SL{e}a_1a_2a_1fa_2a_1a_2,\\
   &y^2=a_2[ea_1e]a_2a_1a_2a_1a_2fa_1a_2,\\
   &y^2=a_2a_1ea_1a_2a_1a_2a_1a_2fa_1a_2.
  \end{aligned}
\]
In the above computations we introduced the notation which should help the reader to follow our transformations. The underlined parts indicate expressions which will be reduced, and parts with small arrows indicate expressions which will be moved to the left/right.

As a conjugation of (B2) we can take
\[(a_2a_1y^{-1}a_2^{-1}ya_1a_2)=y^{-1}a_1(a_2a_1y^{-1}a_2^{-1}ya_1a_2)a_1y^{-1}.\]
By a straightforward computation this gives 
\[y^{-2}=a_2a_1fa_1a_2a_1a_2a_1a_2ea_1a_2,\]
which together with $(\kre{\text{B2}}_1)$ gives $(\kre{\text{B2}}_2)$.

Observe that $(\kre{\text{B2}}_1)$ together with (A1) and $(\kre{\text{A1}}_1)$ imply that we can 
replace $(\kre{\text{B3}})$ for $i\geq 4$ with $(\kre{\text{A1}}_2)$.

Observe also that $(\kre{\text{B2}}_1)$, (A2), $(\kre{\text{B4}}_1)$ and $(\kre{\text{B4}}_2)$
 imply that $(\kre{\text{B3}})$ for $i=1$ is superfluous.
 
We will now show that (D1) is superfluous -- we will need here $(\kre{\text{A2}}_3)$, hence we add this relation to the
statement. Using $(\kre{\text{B2}}_1)$ we substitute for $y^2$.
\[\begin{aligned}
&f=(\SL{\ukre{a_2^{-1}}}a_1^{-1}f^{-1}a_2^{-1}a_1^{-1}a_2^{-1}a_1^{-1}a_2^{-1}\ukre{a_1^{-1}}e^{-1}a_1^{-1}\ukre{a_2^{-1}})
   \SL{\ukre{e}}(\ukre{a_2}a_1e\SR{a_1}a_2a_1a_2a_1a_2fa_1\SR{\ukre{a_2}}),\\
   &\ukre{f}=\ukre{(a_1^{-1}f^{-1}a_2^{-1}a_1^{-1}a_2^{-1}a_1^{-1}a_2^{-1}e^{-1}a_1^{-1})}
   \ukre{(a_1ea_2a_1a_2a_1a_2fa_1)}\ukre{f}.   
  \end{aligned}
\]
{\bf (B1)} If we use (B1) in the form
\[(a_2a_3a_1a_2ya_2^{-1}a_1^{-1}a_3^{-1}a_2^{-1})y^{-1}=y^{-1}(a_2a_3a_1a_2ya_2^{-1}a_1^{-1}a_3^{-1}a_2^{-1}),\]
after rewriting we get $(\kre{\text{B1}})$. Conjugating this relation by $y$ gives
\[\begin{aligned}  
(\kre{\text{B1}}_2)\quad&y(a_2a_3a_1a_2\SL{y}a_2^{-1}a_1^{-1}a_3^{-1}a_2^{-1})\SR{y^{-2}}=(a_2a_3a_1a_2\ukre{y}a_2^{-1}a_1^{
-1 } a_3^{-1}a_2^{-1 } )\SL { \ukre{y^{-1}}},\\
   &y^2(f^{-1}a_3a_1^{-1}f^{-1}a_2^{-1}a_1^{-1}a_3^{-1}a_2^{-1})=(a_2a_3a_1a_2ea_1a_3^{-1}e)y^2.
  \end{aligned}
\]
Now we will show that this relation is superfluous -- it is a consequence of relations
(A1), (A2), $(\kre{\text{A2}}_1)$--$(\kre{\text{A2}}_4)$, $(\kre{\text{B1}})$, $(\kre{\text{B2}}_1)$,
$(\kre{\text{B2}}_2)$, 
$(\kre{\text{B4}}_1)$, $(\kre{\text{B4}}_2)$. We substitute for $y^2$ using  $(\kre{\text{B2}}_1)$ and
$(\kre{\text{B2}}_2)$.
\[\begin{aligned}
   &(\ukre{a_2a_1}ea_1a_2a_1\ukre{a_2}a_1a_2fa_1a_2)(\SL{\ukre{f^{-1}}}a_3a_1^{-1}f^{-1}a_2^{-1}\ukre{a_1^{-1}}a_3^{-1}\ukre{a_2^{-1}})=\\
   &\qquad\qquad\quad=(\ukre{a_2}a_3\ukre{a_1}a_2ea_1a_3^{-1}\SR{\ukre{e}})(a_2^{-1}a_1^{-1}e^{-1}\ukre{a_1^{-1}}a_2^{-1}a_1^{-1}a_2^{-1}a_1^{-1}a_2^{-1}f^{-1}\ukre{a_1^{-1}a_2^{-1}}),\\
   &(ea_1a_2\SL{a_1}\SR{a_1}\ukre{a_2}fa_1a_2)(a_3\SL{\ukre{a_1^{-1}}}f^{-1}a_2^{-1}a_3^{-1})=\\
   &\qquad\qquad\qquad\qquad\qquad\quad=(a_3a_2e\SR{\ukre{a_1}}a_3^{-1})(a_2^{-1}a_1^{-1}e^{-1}\ukre{a_2^{-1}}\SL{a_1^{-1}}a_2^{-1}a_1^{-1}\SR{a_2^{-1}}f^{-1}),\\
   &(\SL{e}a_1a_2fa_1a_2)(\ukre{f}a_3f^{-1}a_2^{-1}a_3^{-1})\SL{\ukre{a_2}}=
   \SR{\ukre{a_2^{-1}}}(a_3a_2ea_3^{-1})(\ukre{e^{-1}}a_2^{-1}a_1^{-1}e^{-1}a_2^{-1}a_1^{-1}\SR{f^{-1}}),\\
   &(a_1a_2fa_1a_2)(\ukre{a_3}\SR{f^{-1}}\SL{a_2^{-1}}\ukre{a_3^{-1}}f)=
   (e^{-1}\ukre{a_3}\SR{a_2}\SL{e}\ukre{a_3^{-1}})(a_2^{-1}a_1^{-1}e^{-1}a_2^{-1}a_1^{-1}),\\   
   &a_1a_2fa_1a_3^{-1}\SR{a_2}f\SR{a_3}=
   \SL{a_3^{-1}}e^{-1}\SL{a_2^{-1}}a_3a_1^{-1}e^{-1}a_2^{-1}a_1^{-1},\\
   &a_2a_3a_1a_2fa_1a_3^{-1}f=
   e^{-1}a_3a_1^{-1}e^{-1}a_2^{-1}a_1^{-1}a_3^{-1}a_2^{-1}.
  \end{aligned}
\]
What we get is $(\kre{\text{B1}})$.\\
{\bf (B6)} If we rewrite (B6) we get $(\kre{\text{B6}}_1)$, and (B6) conjugated by $y$ gives $(\kre{\text{B6}}_2)$.\\
{\bf (B7)} If we use (B7) in the form
\[\begin{aligned}&a_4a_5a_3a_4a_2a_3a_1a_2ya_2^{-1}a_1^{-1}a_3^{-1}a_2^{-1}a_4^{-1}a_3^{-1}a_5^{-1}a_4^{-1}by^{-1}=\\
&\qquad\qquad\qquad\qquad\quad =ba_4a_5a_3a_4a_2a_3a_1a_2ya_2^{-1}a_1^{-1}a_3^{-1}a_2^{-1}a_4^{-1}a_3^{-1}a_5^{-1}a_4^{-1}y^{-1},\end{aligned}\]
after rewriting we get $(\kre{\text{B7}}_1)$. By conjugating this relation by $y^{-1}$, taking inverses 
of both sides and using (D1), we get 
\[\begin{aligned}
&(\SL{a_4a_5a_3a_4a_2a_3a_1a_2}fa_1a_3^{-1}fa_4^{-1}a_3^{-1}a_5^{-1}a_4^{-1})\SL{y^{-2}}cy^2=\\
&\qquad\qquad\qquad\qquad\qquad\qquad\qquad\quad =b(a_4a_5a_3a_4a_2a_3a_1a_2\SR{fa_1a_3^{-1}fa_4^{-1}a_3^{-1}a_5^{-1}a_4^{-1}}),\\
&y^{-2}(ea_1a_3^{-1}ea_4^{-1}a_3^{-1}a_5^{-1}a_4^{-1})c(a_4a_5a_3a_4e^{-1}a_3a_1^{-1}e^{-1})y^2=\\
&\qquad\qquad\qquad\qquad\qquad\quad\quad =(a_2^{-1}a_1^{-1}a_3^{-1}a_2^{-1}a_4^{-1}a_3^{-1}a_5^{-1}a_4^{-1})b(a_4a_5a_3a_4a_2a_3a_1a_2).\\
\end{aligned}\]
This together with $(\kre{\text{B7}}_1)$ gives $(\kre{\text{B7}}_2)$. For further reference observe that using
$(\kre{\text{B1}})$ the above relation can be also rewritten as
\[\begin{aligned}
   (\kre{\text{B7}}_3)\qquad&   
   (a_4a_5a_3a_4e^{-1}a_3a_1^{-1}e^{-1})(a_2^{-1}a_1^{-1}a_3^{-1}a_2^{-1}a_4^{-1}a_3^{-1}a_5^{-1}a_4^{-1})y^{-2}cy^2=\\
&\qquad\qquad\quad =b(a_4a_5a_3a_4e^{-1}a_3a_1^{-1}e^{-1})(a_2^{-1}a_1^{-1}a_3^{-1}a_2^{-1}a_4^{-1}a_3^{-1}a_5^{-1}a_4^{-1}).
  \end{aligned}
\]
Observe that we can use $(\kre{\text{B7}}_1)$ and $(\kre{\text{B6}}_1)$ as definitions of $c$. It is straightforward to check
that the first of these relations imply $(\kre{\text{A3}}_2)$ and $(\kre{\text{A3}}_1)$ for $i=1$. The second one imply 
$(\kre{\text{A3}}_1)$ for $i=3$ and $i\geq 5$.\\
{\bf (B8)} If we rewrite (B8) we get $(\kre{\text{B8}}_1)$ and (B8) conjugated by $y^{-1}$ gives 
$(\kre{\text{B8}}_2)$.
{\bf Further reductions.}
For any $3\leq k\leq g-1$ define
\[z_k=(a_{k-1}a_ka_{k-2}a_{k-1}\cdots a_3a_4e^{-1}a_3a_1^{-1}e^{-1})(a_2^{-1}a_1^{-1}\cdots
 a_{k-1}^{-1}a_{k-2}^{-1}a_{k}^{-1}a_{k-1}^{-1}).\]
Geometrically $z_k$ is the product of crosscap slides $yY_{\mu_k,\alpha_k}^{\pm 1}$, where $\mu_k$ and $\alpha_k$ are circles
 indicated in Figure \ref{r04} (see Section 4 of \cite{StukowSimpSzepPar}), hence on the left of $\mu_k$, conjugation by $z_k$ has the same effect as conjugation by $y$. More precisely,
 \begin{itemize}
  \item[(D2)] $z_ka_1z_k^{-1}=a_1^{-1}$ 
  \item[(D3)] $z_ka_2z_k^{-1}=e^{-1}$\quad for $k\geq 4$,
  \item[(D4)] $z_ka_iz_k^{-1}=a_i$\quad for $3\leq i\leq k-2$,
  \item[(D5)] $z_kbz_k^{-1}=c$\quad for $k\geq 5$,
  \item[(D6)] $z_ky^2z_k^{-1}=y^2$,
  \item[(D7)] $z_kfz_k^{-1}=a_2^{-1}$\quad for $k\geq 4$,
  \item[(D8)] $z_kez_k^{-1}=y^2a_2^{-1}y^{-2}$\quad for $k\geq 4$,
  \item[(D9)] $z_kcz_k^{-1}=y^2by^{-2}$ \quad for $k\geq 5$.
 \end{itemize}
 Relations (D2)--(D4) are straightforward consequences of (A1), (A2), ($\kre{\text{A1}}_1$), ($\kre{\text{A2}}_1$), 
 ($\kre{\text{A2}}_2$). For (D5) we need additionally (A3), ($\kre{\text{A3}}_1$) and ($\kre{\text{B7}}_1$).
 
 Let us prove (D6) -- we will use (A1), ($\kre{\text{A1}}_1$), ($\kre{\text{B1}}$), ($\kre{\text{B3}}$) and ($\kre{\text{B1}}_2$) (hence we need
 all relations that we used to reduce ($\kre{\text{B1}}_2$)).
 \[\begin{aligned}
    z_ky^2&=(a_{k-1}a_ka_{k-2}a_{k-1}\cdots a_3a_4e^{-1}a_3a_1^{-1}e^{-1})(a_2^{-1}a_1^{-1}\cdots
 a_{k-1}^{-1}a_{k-2}^{-1}a_{k}^{-1}a_{k-1}^{-1})y^2=\\
 &=(a_{k-1}a_{k}\cdots a_3a_4)[e^{-1}a_3a_1^{-1}e^{-1}a_2^{-1}a_1^{-1}a_3^{-1}a_2^{-1}]y^2
(a_4^{-1}a_3^{-1}\cdots a_{k}^{-1}a_{k-1}^{-1})=\\
 &=(a_{k-1}a_{k}\cdots a_3a_4)y^2[e^{-1}a_3a_1^{-1}e^{-1}a_2^{-1}a_1^{-1}a_3^{-1}a_2^{-1}]
(a_4^{-1}a_3^{-1}\cdots a_{k}^{-1}a_{k-1}^{-1})=y^2z_k.
   \end{aligned}
\]
 Now we will prove (D7) -- we will use (A1), (A2), ($\kre{\text{A1}}_2$), ($\kre{\text{A2}}_4$), ($\kre{\text{B1}}$).
\[\begin{aligned}
z_kf&=
   (a_{k-1}a_{k}\cdots a_4a_5)a_3a_4[e^{-1}a_3a_1^{-1}e^{-1}a_2^{-1}a_1^{-1}a_3^{-1}a_2^{-1}]a_4^{-1}a_3^{-1}
   (a_5^{-1}a_4^{-1}\cdots a_{k}^{-1}a_{k-1}^{-1})f=\\  
   &=(a_{k-1}a_{k}\cdots a_4a_5)a_3a_4[a_2a_3a_1a_2fa_1a_3^{-1}f]a_4^{-1}a_3^{-1}\SL{f}
   (a_5^{-1}a_4^{-1}\cdots a_{k}^{-1}a_{k-1}^{-1})=\\  
   &=(a_{k-1}a_{k}\cdots a_4a_5)a_2^{-1}a_3a_4[a_2a_3a_1a_2fa_1a_3^{-1}f]a_4^{-1}a_3^{-1}
   (a_5^{-1}a_4^{-1}\cdots a_{k}^{-1}a_{k-1}^{-1})=\\  
   &=a_2^{-1}(a_{k-1}a_{k}\cdots a_3a_4)[e^{-1}a_3a_1^{-1}e^{-1}a_2^{-1}a_1^{-1}a_3^{-1}a_2^{-1}](a_4^{-1}a_3^{-1}
   \cdots a_{k}^{-1}a_{k-1}^{-1})=a_2^{-1}z_k.
  \end{aligned}
\]
Relation (D8) is a consequence of (D6), (D7) and (D1). Finally, (D9) is a consequence of ($\kre{\text{B7}}_3$) and (D6) 
(hence we need ($\kre{\text{B7}}_2$)).

Relations (D2)--(D9) imply that
\begin{itemize}
 \item $(\kre{\text{A4}})$ is superfluous if $g\geq 7$,
 \item $(\kre{\text{A5}})$ is superfluous if $g\geq 7$,
 \item $(\kre{\text{A6}})$ is superfluous if $g\geq 9$, 
 \item $(\kre{\text{B6}}_2)$ is superfluous if $g\geq 6$, 
 \item $(\kre{\text{B8}}_2)$ is superfluous if $g\geq 7$.
\end{itemize}
Moreover, if $g\geq 8$, relations ($\kre{\text{A8a}}$) and ($\kre{\text{A8b}}$) for $i<\frac{g-4}{2}$
are consequences of relation (A8). Hence we can remove all these relations together with generators
$\kre{b}_0,\ldots,\kre{b}_{\frac{g-8}{2}}$ and instead add the relation
\[\kre{b}_{i}=z_{g-1}b_iz_{g-1}^{-1}\quad\text{for $i=\frac{g-6}{2}, \frac{g-4}{2}$}.\]
This is exactly ($\kre{\text{A7c}}$).
\end{proof}
\begin{tw}\label{PresTwist2}
 If $g\geq 5$ is odd, then the group ${\cal T}(N_{g,0})$ is isomorphic to the quotient of the group ${\cal T}(N_{g,1})$ with
presentation given in Theorem \ref{PresTwist1} obtained by adding a generator $\ro$ and relations
\begin{itemize}
 \item[(C1a)] $(a_1a_2\cdots a_{g-1})^g=\ro$,
 \item[$(\kre{\text{C1a}})$] $(a_1^{-1}e^{-1}a_3\cdots a_{g-1})^g=y^2\ro$,
 \item[(C2)] $a_i\ro=\ro a_i$\quad for $i=1,2,\ldots,g-1$,
 \item[$(\kre{\text{C2}})$] $\ro e=f\ro$,
 \item[$(\kre{\text{C5}})$] $\ro y^2=y^{-2}\ro$,
 \item[(C3)] $\ro^2=1$,
 \item[$(\kre{\text{C4a}})$] $(a_2a_3\cdots a_{g-1}e^{-1}a_3\cdots a_{g-1})^{\frac{g-1}{2}}=1$.
\end{itemize}
Moreover, relations $\text{(}\kre{\text{A1}}_\text{2}\text{)}$, $\text{(}\kre{\text{B2}}_\text{2}\text{)}$, $\text{(}\kre{\text{B4}}_\text{2}\text{)}$ are superfluous.

If $g\geq 4$ is even, then the group ${\cal T}(N_{g,0})$ is isomorphic to the quotient of the group ${\cal T}(N_{g,1})$ with
presentation given in Theorem \ref{PresTwist1} obtained by adding a generator $\kre{\ro}$ and relations
\begin{itemize}
 \item[(C1b)] $(a_1a_2\cdots a_{g-1})^g=1$,
 \item[$(\kre{\text{C2}}_1)$] $\kre{\ro}a_1=a_1^{-1}\kre{\ro}$,
 \item[$(\kre{\text{C2}}_2)$] $\kre{\ro}a_i=a_i \kre{\ro}$\quad for $i=3,\ldots,g-1$,
 \item[$(\kre{\text{C2}}_3)$] $\kre{\ro}a_2=e^{-1} \kre{\ro}$,
 \item[$(\kre{\text{C5}})$] $\kre{\ro}y^2=y^{-2}\kre{\ro} $,
 \item[$(\kre{\text{C3}})$] $\kre{\ro}^2=1$,
 \item[$(\kre{\text{C4}})$] $(\kre{\ro} a_2a_3\cdots a_{g-1})^{g-1}=1$.
\end{itemize}
Moreover, relations $\text{(}\kre{\text{A1}}_\text{1}\text{)}$, $(\kre{\text{A2}}_\text{1}\text{)}$, $\text{(}\kre{\text{A2}}_\text{2}\text{)}$ are superfluous.
\end{tw}
\begin{proof}
 We follow the lines of the proof of Theorem \ref{PresTwist1}, but as a starting point we now have Theorem
\ref{SimParSzep2}. Moreover, it is convenient to add relations (C2) and (C5), so in particular (C4a) and (C4b) are equivalent to (C4) (see Theorem \ref{Uw:F}). Generator $\ro$ yields two additional generators for ${\cal T}(N_{g,0})$, namely $\ro, y\ro y^{-1}$ if $g$ is odd and $\kre{\ro}=y\ro, \ro y^{-1}$ if $g$ is even. 

Suppose first that $g$ is odd. Then (C5) and its conjugate by $y^{-1}$ rewrite as
\[\begin{aligned}
  &y^2\ro=y\ro y^{-1}, \\
  &y^2\ro y^2=\ro.
  \end{aligned}
\]
The first relation implies that we can remove generator $y\ro y^{-1}$ -- we will do this silently from now on. The
second one gives $(\kre{\text{C5}})$.

Relations (C1a), (C2), (C3) does not need rewriting, and if we conjugate them by $y$ we get respectively $(\kre{\text{C1a}})$,
$(\kre{\text{C2}})$ (we use here $(\kre{\text{D1}})$, hence also $(\kre{\text{A2}}_3)$) and relation equivalent to $(\kre{\text{C5}})$.

Relation (C4a) and its conjugate by $y$ rewrite respectively as 
\[\begin{aligned}
    &(f^{-1}a_3\cdots a_{g-1}a_2a_3\cdots a_{g-1})^{\frac{g-1}{2}}=1,\\
    &(a_2a_3\cdots a_{g-1}e^{-1}a_3\cdots a_{g-1})^{\frac{g-1}{2}}=1.
  \end{aligned}
\]
The second relation is $(\kre{\text{C4a}})$, and if we conjugate it by $\ro$, by (C2) and $(\kre{\text{C2}})$
we get the first one.

Finally, observe that if we conjugate relations $(\kre{\text{A1}}_1)$, $(\kre{\text{B2}}_1)$,
$(\kre{\text{B4}}_1)$ by $\ro$ we get respectively $(\kre{\text{A1}}_2)$, $(\kre{\text{B2}}_2)$,
$(\kre{\text{B4}}_2)$.

Now assume that $g$ is even, hence $\kre{\ro}=y\ro\in {\cal T}(N_{g,0})$. Relation (C5) and its conjugate by $y$ rewrite
as
\[\begin{aligned}
   &y\ro=\ro y^{-1},\\
   &y^2(y\ro)=(y\ro)y^{-2}.
  \end{aligned}
\]
The first relation implies that we can remove generator $\ro y^{-1}$ -- we will do this silently from now on. The
second one gives $(\kre{\text{C5}})$.

If we rewrite relation (C2) we get relations $(\kre{\text{C2}}_1)$--$(\kre{\text{C2}}_3)$. 

Relations (C1b), (C3) and (C4) rewrite respectively as (C1b), $(\kre{\text{C3}})$ and $(\kre{\text{C4}})$. Their conjugates by $y^{\pm 1}$ are superfluous since, by $(\kre{\text{C2}}_1)$--$(\kre{\text{C2}}_3)$, they are the same as conjugates by
$\kre{\ro}$.

Finally, observe that if we conjugate relations 
(A1), (A2) by $\kre{\ro}$  we get respectively $(\kre{\text{A1}}_1)$, $(\kre{\text{A2}}_1)$--$(\kre{\text{A2}}_2)$.
\end{proof}
\begin{uw}
 Observe that relations $(\kre{\text{B2}}_1)$ and (C1a), $(\kre{\text{C4}})$ allows to remove $y^2$ and $\ro, \kre{\ro}$
from 
 the generating sets, hence the generating sets of the presentations given by Theorems \ref{PresTwist1} and \ref{PresTwist2}
are
 really Dehn twists about nonseparating circles.
 \end{uw}
\section{Geometric interpretation}
We devote this last section to the geometric interpretation of relations
obtained in Theorem \ref{PresTwist1}.

Relations (A1), (A3), (A9a), (A9b), $(\kre{\text{A1}}_1)$, $(\kre{\text{A1}}_2)$,
$(\kre{\text{A3}}_1)$, $(\kre{\text{A3}}_2)$, $(\kre{\text{B3}})$,
$(\kre{\text{B4}}_1)$, $(\kre{\text{B4}}_2)$, $(\kre{\text{B7}}_2)$,
$(\kre{\text{A9a}})$, $(\kre{\text{A9b}})$ are standard commutativity relations
between Dehn twists with disjoint supports.

Relations (A2), (A4), $(\kre{\text{A2}}_1)$--$(\kre{\text{A2}}_3)$,
$(\kre{\text{A4}})$ are standard braid relations between Dehn twists about
circles intersecting in one point.

Relations (A5), (A6), (A8), $(\kre{\text{A5}})$, $(\kre{\text{A6}})$,
$(\kre{\text{A8a}})$, $(\kre{\text{A8b}})$ came from Matsumoto
\cite{MatsumotoPres} presentation of mapping class group of an orientable
surface. They have simple interpretation as relations between fundamental
elements of Artin groups -- for details see \cite{MatsumotoPres} and
\cite{ParLab}. 

Relations $(\kre{\text{B7}}_1)$ and $(\kre{\text{A7c}})$ are simple conjugation
relations of the form $t_{f(\alpha)}=ft_{\alpha}f^{-1}$, where $t_{\alpha}$ is
the twist about a circle $\alpha$.

Relations $(\kre{\text{B6}}_2)$ and $(\kre{\text{B8}}_2)$ are conjugates (by
$y^{\pm 1}$) of $(\kre{\text{B6}}_1)$ and $(\kre{\text{B8}}_1)$ respectively,
and $(\kre{\text{B2}}_2)$ is equivalent to the conjugation of
$(\kre{\text{B2}}_1)$, hence we are left with four interesting relations:
$(\kre{\text{B1}})$, $(\kre{\text{B2}}_1)$, $(\kre{\text{B6}}_1)$ and
$(\kre{\text{B8}}_1)$.

Relation $(\kre{\text{B1}})$ can be rewritten in a slightly more symmetric form
\[
(a_2ea_1)a_3^{-1} (a_2ea_1)a_3(a_2fa_1)a_3^{-1} (a_2fa_1)a_3=1.
\]
This is a relation between five Dehn twists $a_1,a_2,a_3,e,f$ illustrated in
Figures \ref{r01} and \ref{r07}.

Relation $(\kre{\text{B2}}_1)$ can be rewritten as
\[y^2=(a_2ea_1)^2(a_2fa_1)^2.\]
This is a relation between five twists $a_1,a_2,e,f,y^2$ illustrated in Figures
\ref{r01} and \ref{r07}.

Relation $(\kre{\text{B6}}_1)$ is a relation between four Dehn twists 
\[b,\ c,\ f'=(a_1a_2a_3)f^{-1}(a_1a_2a_3)^{-1},\
e'=(a_3a_2)^{-1}e^{-1}(a_3a_2),\]
illustrated in Figures \ref{r01}, \ref{r07} and \ref{r06}
\begin{figure}[h]
\begin{center}
\includegraphics[width=0.47\textwidth]{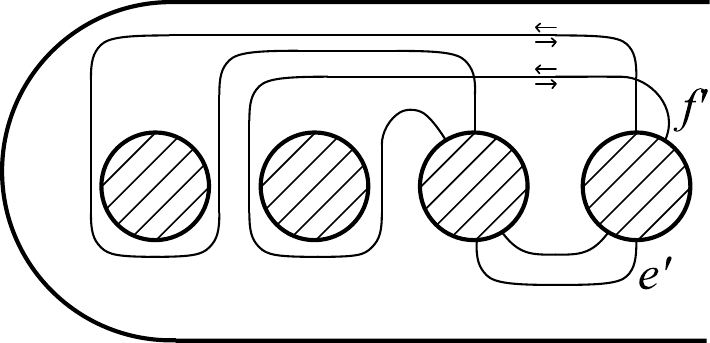}
\caption{Dehn twists $e'$ and $f'$.}\label{r06} %
\end{center}
\end{figure}

Finally, relation $(\kre{\text{B8}}_1)$ is a relation between six Dehn twists 
\[\begin{aligned}
&c'=(a_1ea_3^{-1}a_4^{-1})c(a_1ea_3^{-1}a_4^{-1})^{-1},\ b'=(a_4a_3a_2a_1)^{-1}b^{-1}(a_4a_3a_2a_1),\\
&a_4,\ a'=(a_3^{-1}ea_2a_3)^{-1}a_4(a_3^{-1}ea_2a_3),\ a_2,\ e.
\end{aligned}\]
illustrated in Figures \ref{r01}, \ref{r07} and \ref{r05}.
\begin{figure}[h]
\begin{center}
\includegraphics[width=0.74\textwidth]{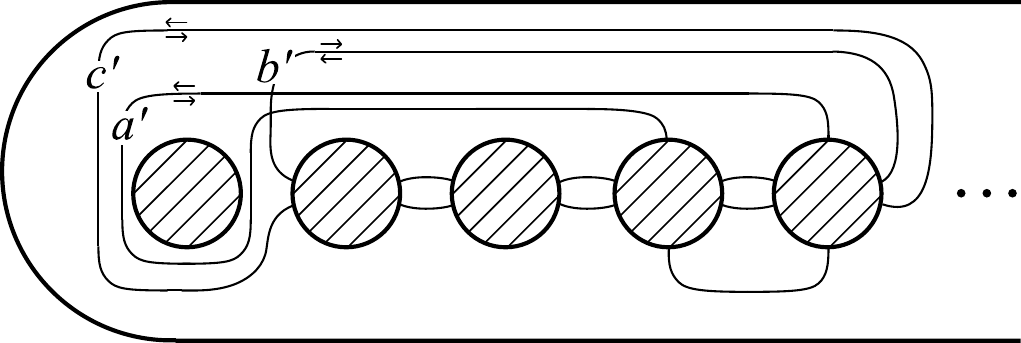}
\caption{Dehn twists $a'$, $b'$ and $c'$.}\label{r05} %
\end{center}
\end{figure}
\bibliographystyle{acm}

\end{document}